\documentclass[12pt,reqno]{article}
\usepackage{amssymb,amscd,amsmath,amsthm,color}
\newtheorem{theorem}{Theorem}[section]
\newtheorem{lemma}[theorem]{Lemma}
\newtheorem{cor}[theorem]{Corollary}

\usepackage{enumerate,amssymb}
\theoremstyle{definition}

\theoremstyle{remark}

\numberwithin{equation}{section}

\def\bC{\mathbb{C}}

\def\bM{\mathbb{M}}

\def\bR{\mathbb{R}}

\def\bH{\mathbb{H}}

\begin{document}
\baselineskip=15pt

\title{ Positive  matrices partitioned into a small number of Hermitian blocks }

\author{Jean-Christophe Bourin{\footnote{Supported by ANR 2011-BS01-008-01.}},  Eun-Young Lee{\footnote{Research  supported by  Basic Science Research Program
through the National Research Foundation of Korea (NRF) funded by the Ministry of Education,
Science and Technology (2010-0003520)}}, Minghua Lin}

\date{ }

\maketitle

\begin{abstract}
\noindent Positive semidefinite matrices partitioned into a small number of Hermitian blocks have a remarkable property. Such a matrix  may be written in a  simple way from the sum  of its diagonal blocks: the full matrix  is a kind of average of  copies of the sum of the diagonal blocks.  This  entails several eigenvalue inequalities. The proofs use a decomposition lemma for positive  matrices, isometries with complex entries, and the Pauli matrices.
\end{abstract}

{\small\noindent

\noindent
AMS subjects classification 2010:  15A60, 47A30,  15A42.}

\section{Introduction}

\noindent
Positive  matrices partitioned into blocks frequently occur as basic tools or for their own interest in matrix analysis and mathematical physics. For instance,   to define the geometric mean
of two $n\times n$ matrices $A,B\in\bM_n^+$, the space of positive semidefinite matrices, one consider the class of block-matrices 
\begin{equation}\label{block1}
\begin{bmatrix}
A&X\\ X&B
\end{bmatrix}
\end{equation}
belonging to $\bM_{2n}^+$ (hence $X$ is Hermitian). The geometric mean of $A$ and $B$ is then characterized as the largest possible $X$ such that  \eqref{block1} is positive. Positive matrices $H$ of the form \eqref{block1} enjoy a remarkable property: for all symmetric norms
\begin{equation}\label{norm1}
\left\| H\right\|
\le \| A+B \|.
\end{equation}
This  says that we have a majorisation  between $H$ and the sum of its diagonal block,
\begin{equation}\label{major1}
\sum_{j=1}^k \lambda_j(H) \le \sum_{j=1}^k\lambda_j(A+B), \qquad k=1,\ldots,2n
\end{equation}
where $\lambda_j(T)$  is the $j$th eigenvalue of $T\in\bM_m^+$ (in decreasing order, with the  convention $\lambda_j(T)=0$ for $j>m$).

Another typical example of positive matrices written in blocks are formed by tensor products. Indeed, 
 the tensor product $A\otimes B$ of $A\in\bM_{\beta}$ with $B\in\bM_n$ can be identified with an element of $\bM_{\beta}(\bM_n)=\bM_{\beta n}$. Starting with positive matrices in  $\bM_{\beta}^+$ and $\bM_n^+$ we then get a matrix in $\bM_{\beta n}^+$ partitioned in blocks in $\bM_n$. In quantum physics, sums of tensor products of positive semi-definite (with trace one) occur as so-called separable states. In this setting, the sum of the diagonal block is called the partial trace (with respect to $\bM_{\beta}$).  Hiroshima \cite{H} proved a beautiful extension of \eqref{norm1}-\eqref{major1}:

\vskip 10pt
\begin{theorem}\label{Hiroshima} Let $H=[A_{s,t}]\in \bM_{\alpha n}^+$ be partitioned into $\alpha\times \alpha$  Hermitian blocks in $\bM_n$ and let $\Delta=\sum_{s=1}^{\alpha}{A_{s,s}}$ be its partial trace.   Then,  we have
\begin{equation*}
\left\|H \right\|\le \left\| \Delta\right\|
\end{equation*} for all symmetric norms.
\end{theorem}

\vskip 10pt
This result seems to be not so well-known among matrix-functional analysts. We recently  rediscovered it and actually obtained a stronger  decomposition theorem \cite{BLL3}. For small partitions,  when $\alpha\in\{2,3,4\}$,  special proofs are  available that differ from the general case in two related ways. First, these proofs are simpler (especially for $\alpha=2$) but also yield a much sharper decomposition than the one we can obtain in the general case. Secondly, and this is rather surprising, even though we consider a positive matrix with {\it real} entries, its decomposition  involves some {\it complex} matrices. This note is concerned with these special decompositions for small partitions.

In the next section we will see what can be said for $\alpha=2$. This situation was already implicitly covered in some proofs given in our first note \cite{BLL1}. Section 3 deals with the case $\alpha=4$ (and as a byproduct the case $\alpha=3$). 

\section{Two by two blocks}

\noindent
For partitions of positive matrices,
 the diagonal blocks  play a special role. This is apparent in a rather striking  decomposition  due  to  the two first authors \cite{BL1}:

\vskip 10pt
\begin{lemma} \label{BL-lemma} For every matrix in  $\bM_{n+m}^+$ partitioned into blocks, we have a decomposition
\begin{equation*}
\begin{bmatrix} A &X \\
X^* &B\end{bmatrix} = U
\begin{bmatrix} A &0 \\
0 &0\end{bmatrix} U^* +
V\begin{bmatrix} 0 &0 \\
0 &B\end{bmatrix} V^*
\end{equation*}
for some unitaries $U,\,V\in  \bM_{n+m}$.
\end{lemma}

\vskip 10pt
This decomposition turned out to be an efficient tool and it also plays a major role below.
A proof and several consequences can be found in \cite{BL1} and \cite{BH1}. 
Of course,  $\bM_n$ is the algebra of $n\times n$  matrices with real or complex entries, and $\bM_n^+$ is the positive  part. That is, $\bM_n$ may stand either for $\bM_n(\bR)$, the matrices with real entries, or for  $\bM_n(\bC)$, those with complex entries. The situation is different
in the next statement, where complex entries seem unavoidable. 

\vskip 10pt
\begin{theorem}\label{thm-four} Given any matrix in $\bM_{2n}^+(\bC)$ partitioned into blocks in $\bM_n(\bC)$ with Hermitian off-diagonal blocks, we have
\begin{equation*}
\begin{bmatrix} A &X \\
X &B\end{bmatrix}= \frac{1}{2}\left\{ U(A+B)U^* +V(A+B)V^*\right\}
\end{equation*} for some isometries $U,V\in\bM_{2n,n}(\bC)$.
\end{theorem}

\vskip 10pt
Here $\bM_{p,q}(\bC)$ denote the space of $p$ rows and $q$ columns matrices with complex entries, and $V\in\bM_{p,q}(\bC)$ is  an isometry if $p\ge q$ and $V^*V=I_q$. Even for a matrix in $\bM_{2n}^+(\bR)$, it seems essential to use isometries with complex entries.

Theorem \ref{thm-four} is implicit in \cite{BLL1}. We detail here how it follows from Lemma \ref{BL-lemma}.

\vskip 10pt
\begin{proof}Taking the unitary matrix 
$$W=\frac{1}{\sqrt{2}}\begin{bmatrix} -iI&iI \\I &I\end{bmatrix},$$
 where $I$ is the identity of $\bM_n$, then
  \begin{equation*}
 W^*\begin{bmatrix} A &X \\
X &B\end{bmatrix}W= \frac{1}{2}\begin{bmatrix} A+B & \ast\\
\ast & A+B\end{bmatrix}
\end{equation*}
where $\ast$ stands for unspecified entries. By Lemma \ref{BL-lemma}, there are two unitaries $U,V\in\bM_{2n}$ partitioned into  equally sized matrices,
$$U=\begin{bmatrix} U_{11}&   U_{12}  \\ U_{21}  &   U_{22}\end{bmatrix},\qquad  V=\begin{bmatrix} V_{11}&   V_{12}  \\ V_{21}  &   V_{22}\end{bmatrix}$$
 such that
 $$
 \frac{1}{2}\begin{bmatrix} A+B & \ast\\
\ast & A+B\end{bmatrix}= \frac{1}{2}\left\{U\begin{bmatrix} A+B & 0\\0 &0\end{bmatrix}U^*+V\begin{bmatrix} 0& 0\\0 & A+B\end{bmatrix}V^*\right\}.
$$
 Therefore
$$  \frac{1}{2}\begin{bmatrix} A+B & \ast\\
\ast & A+B\end{bmatrix}=\frac{1}{2}\left\{ \widetilde{U}(A+B)\widetilde{U}^* +\widetilde{V}(A+B)\widetilde{V}^*\right\}$$
 where $$\widetilde{U}=\begin{bmatrix} U_{11} \\U_{21}\end{bmatrix}\quad {\rm and}\quad  \widetilde{V}=\begin{bmatrix} V_{11} \\V_{21}\end{bmatrix}$$
are isometries. The proof is complete by assigning $W\widetilde{U}$, $W\widetilde{V}$ to new isometries $U, V$, respectively. 
\end{proof}

\vskip 10pt
Theorem \ref{thm-four}
  yields \eqref{norm1}. As a consequence of this inequality we have a refinement of a well-known determinantal inequality,
$$
 \det(I+A)\det(I+B)\ge \det(I+A+B)
$$
for all $A,B\in\bM^+_n$.

\begin{cor} Let $A,B\in \bM_n^+$. For any Hermitian $X\in\bM_n$  such that 
$H=\begin{bmatrix} A&X \\X&B\end{bmatrix}$
is positive semi-definite, we have 
\begin{equation*}
\det(I+A)\det(I+B) \ge \det(I+H) \ge \det(I+A+B).
\end{equation*}
\end{cor}

\vskip 10pt
Here $I$ denotes both the identity of $\bM_n$ and $\bM_{2n}$.
Note that equality obviously occurs in the first inequality when $X=0$, and equality occurs in the second inequality when $AB=BA$ and $X=A^{1/2}B^{1/2}$.

\vskip 10pt
\begin{proof} The left inequality is a special case of Fisher's inequality, 
$$
\det X\det Y \ge \det \begin{bmatrix} X&Z \\ Z^*&Y \end{bmatrix}
$$
for any partitioned positive semi-definite matrix. The second inequality follows from  \eqref{norm1}. Indeed,
the majorisation $S\prec T$ in $\bM_n^+$ entails the trace inequality
\begin{equation}\label{basic-concave}
{\mathrm{Tr\,}} f(S) \ge {\mathrm{Tr\,}} f(T)
\end{equation}
for all concave functions $f(t)$ defined on $[0,\infty)$. using \eqref{basic-concave} with $f(t)=\log(1+t)$ and the relation $H\prec A+B$ we have
\begin{align*}
\det(I+H) &=\exp{\mathrm{Tr\,}} \log(I+H) \\
&\ge\exp{\mathrm{Tr\,}} \log(I+((A+B)\oplus 0_n)) \\
&=\det(I+A+B).
\end{align*}
\end{proof}

\vskip 10pt

Theorem \ref{thm-four} says more than the eigenvalue majorisation \eqref{major1}. We have a few other eigenvalue inequalities as follows.

\vskip 10pt
\begin{cor}\label{cor-1} Let $H=\begin{bmatrix} A&X \\ X&B\end{bmatrix}\in\bM_{2n}^+$ be partitioned into   Hermitian blocks in $\bM_n$.   Then,  we have
\begin{equation*}
\lambda_{1+2 k}(H) \le \lambda_{1+k}( A+B)
\end{equation*} 
 for all $k=0,\ldots,n-1$.
\end{cor}

\vskip 10pt
\begin{proof} Together with Theorem \ref{thm-four}, the alleged inequalities  follow immediately from a simple fact, Weyl's theorem: if $Y, Z \in\bM_m$ are Hermitian,  then
\[\lambda_{r+s+1}(Y+Z)\le \lambda_{r+1}(Y) + \lambda_{s+1}(Z)\]
for all nonnegative integers $r,s$ such that $r + s\le m-1$.
\end{proof}

\vskip 10pt
\begin{cor}
 Let $S,T\in\bM_n$ be Hermitian. Then,
 \begin{equation*}
\| T^2 + ST^2S\| \le \| T^2 + TS^2T\|
\end{equation*}
for all symmetric norms, and
\begin{equation*}
\lambda_{1+2k}( T^2 + ST^2S) \le
 \lambda_{1+k}( T^2 + TS^2T)
\end{equation*} for all $k=0,\ldots,n-1$.
\end{cor}

\begin{proof}
 The nonzero eigenvalues of $T^2+ST^2S=\begin{bmatrix}T& ST\end{bmatrix}\begin{bmatrix}T& ST\end{bmatrix}^*$ are the same as those of $$\begin{bmatrix}T& ST\end{bmatrix}^*\begin{bmatrix}T& ST\end{bmatrix}=\begin{bmatrix} T^2 & TST    \\ TST &  TS^2T\end{bmatrix}.$$
This block-matrix  is of course positive and has its off-diagonal blocks Hermitian. Therefore, the norm inequality follows from \eqref{norm1}, and the eigenvalue inequalities from Corollary \ref{cor-1}. The norm inequality was first observed in \cite{FL}. 
\end{proof}

\section{Quaternions and 4-by-4 blocks }

\noindent
Theorem \ref{thm-four} refines Hiroshima's theorem in case of two by two blocks. Some interesting new eigenvalue inequalities are obtained.  How to get a similar result for partitions into a larger number of blocks ? The question whether a positive block-matrix $H$ in $\bM^+_{3n}$, 
$$
H=\begin{bmatrix}
A&X&Y \\ X&B&Z \\Y&Z&C
\end{bmatrix}
$$
with Hermitian off-diagonal blocks $X,Y,Z$, can be decomposed as
$$
H=\frac{1}{3}\left\{ U\Delta U^*+ V\Delta V^* +W\Delta W^*\right\}
$$
where $\Delta=A+B+C$ and $U,V,W$ are isometries, is a difficult one. However, we will give a rather satisfactory answer by considering direct sums.  We have been unable to find any direct proof for partitions in 3-by-3 blocks. The key idea was then to introduce quaternions and to deal with 4-by-4 partitions. This approach leads to the following theorem.

\vskip 10pt
\begin{theorem}\label{thm-quaternion} Let $H=[A_{s,t}]\in \bM_{\beta n}^+(\bC)$ be partitioned into Hermitian  blocks in $\bM_n(\bC)$  with $\beta\in\{3,4\}$ and let $\Delta=\sum_{s=1}^{\beta}A_{s,s}$ be the sum of its diagonal blocks.Then,
\begin{equation*}
H\oplus H =\frac{1}{4}\sum_{k=1}^4 V_k\left(\Delta\oplus\Delta\right)V_k^*
\end{equation*} for some isometries $V_k\in\bM_{2\beta n,2n}(\bC)$, $k=1,2,3,4$.
\end{theorem}

\vskip 10pt
Note that, for $\alpha=\beta\in\{3,4\}$, Theorem \ref{thm-quaternion} considerably improves Theorem \ref{Hiroshima}. Indeed, Theorem  \ref{thm-quaternion} implies the majorisation $\| H\oplus H \| \le \| \Delta \oplus \Delta \|$ which is equivalent to the majorisation of Theorem \ref{Hiroshima},  $\| H \| \le \| \Delta\|$.

Likewise for Theorem \ref{thm-four}, we must consider isometries with complex entries, even for a full matrix $H$ with real entries. In \cite{BLL3} we will develop a real approach for real matrices. The isometries are then with real coefficients, but the proof is more intricate and the result is not so simple since it requires direct sums of sixteen copies:
we obtain a decomposition of $\oplus^{16} H$ in term of $\oplus^{16} \Delta$.

Before turning to the proof, we recall some facts about quaternions.

The algebra $\bH$ of quaternions is an associative real division algebra of dimension four containing $\bC$ as a sub-algebra.
Quaternions $q$ are usually written as
$$
q=a+bi+cj +dk
$$
with $a,b,c,d\in\bR$ and $a+bi\in\bC$. The quaternion units $1,i,j,k$ satisfy
$$
i^2=j^2=k^2=ijk=-1.
$$
The algebra $\bH$ can be represented as the real sub-algebra of $\bM_2$ consisting of matrices of the form 
$$
\begin{pmatrix} z&-\overline{w} \\
w&\overline{z}
\end{pmatrix}
$$
by the identification map
$$
a+bi+cj+dk \mapsto \begin{pmatrix} a+bi & ic-d \\   ic+d & a-ib\end{pmatrix}.
$$
The quaternion units $1,i,j,k$ are then represented by the matrices (related to the Pauli matrices),
\begin{equation}\label{units}
\begin{pmatrix} 1& 0\\
0& 1
\end{pmatrix}, \quad 
\begin{pmatrix} i& 0\\
0& -i
\end{pmatrix}, \quad 
\begin{pmatrix}0 & i\\
i& 0
\end{pmatrix}, \quad 
\begin{pmatrix} 0& -1\\
1& 0
\end{pmatrix}
\end{equation}
that we will use in the following proof of Theorem \ref{thm-quaternion}. 

We will work with matrices in $\bM_{8n}$ partitioned in 4-by-4 blocks in $\bM_{2n}$.

\begin{proof} It suffices to consider the case $\beta=4$, the case $\beta=3$ follows by completing $H$ with some zero colums and rows.

First, replace the positive block matrix $H=[A_{s,t}]$ where $1\le s,t,\le 4$
and all blocks are Hermitian by a bigger one in which each block in counted two times :
$$G= [G_{s,t}]:= [A_{s,t}\oplus A_{s,t}].$$
Thus $G\in\bM_{8n}(\bC)$ is written in 4-by-4 blocks in $\bM_{2n}(\bC)$. Then perform a unitary congruence with the matrix
$$W =  E_1\oplus  E_2\oplus E_3 \oplus E_4$$
where the $E_i$ are the  analogues of  quaternion units, that is, with $I$ the identity of $\bM_n(\bC)$,
$$
E_1 = \begin{bmatrix}I&0 \\ 0&I\end{bmatrix},\quad
E_2 = \begin{bmatrix} iI&0 \\ 0&-iI\end{bmatrix},\quad
E_3 =\begin{bmatrix} 0&iI \\ iI&0\end{bmatrix},\quad
E_4 =\begin{bmatrix}0& -I \\ I&0\end{bmatrix}.
$$
Note that $E_sE_t^*$ is skew-Hermitian whenever $s\neq t$.
A direct matrix computation then shows  that the block matrix
$$
\Omega:=WGW^*=[ \Omega_{s,t}]
$$
has the following property for its off-diagonal blocks :
For $1\le s<t\le 4$
$$
\Omega_{s,t}=-\Omega_{t,s}.
$$
Using this property we compute the unitary congruence implemented by
$$
R_2=\frac{1}{2}\begin{bmatrix} 1&1&1&1 \\
1&-1& 1&-1 \\
1&1&-1&-1 \\
1&-1& -1&1 \\
\end{bmatrix}
\otimes\begin{bmatrix} I&0 \\ 0&I
\end{bmatrix}
$$
and we observe that
$
R_2\Omega R_2^*
$
has  its four diagonal blocks $(R_2\Omega R_2^*)_{k,k}$, $1\le k\le 4$, all equal to the matrix $D\in\bM_{2n}(\bC)$,
$$
D =\frac{1}{4}\sum_{s=1}^4 A_{s,s}\oplus A_{s,s}.
$$
Let $\Gamma=D\oplus 0_{6n}\in\bM_{8n}$.
Thanks to the decomposition  of Lemma \ref{BL-lemma},  there exist some unitaries $U_i\in\bM_{8n}(\bC)$, $1\le i\le 4$, such that
$$
\Omega=\sum_ {i=1}^4 U_i \Gamma U_i^*.
$$
That is, since $\Omega$ is unitarily equivalent to $H\oplus H$, and $\Gamma=WD W^*$ for some isometry $W\in\bM_{8n, 2n}(\bC)$,
$$
H\oplus H =   \sum_{s=1}^4 V_kDV_k^*
$$
for some isometries $V_k\in\bM_{8n,2n}(\bC)$. Since $D=\frac{1}{4}\Delta\oplus\Delta$, the proof is complete.
\end{proof}

\vskip 10pt
In the same vein as in Section 2, we have the following consequences. 

\vskip 10pt
\begin{cor} Let $H=[A_{s,t}]\in \bM_{\beta n}^+$ be written in  Hermitian blocks in $\bM_n$  with $\beta\in\{3,4\}$ and let $\Delta=\sum_{s=1}^{\beta}A_{s,s}$ be the sum of its diagonal blocks. Then,
\begin{equation*}
\prod_{s=1}^\beta\det(I+A_{ss}) \ge \det(I+H) \ge \det\left(I+\sum_{s=1}^\beta A_{ss}\right).
\end{equation*}
\end{cor} 
\vskip 10pt

\vskip 10pt
\begin{cor}\label{cor4}  Let $H=[A_{s,t}]\in \bM_{\beta n}^+$ be written in  Hermitian blocks in $\bM_n$  with $\beta\in\{3,4\}$ and let $\Delta=\sum_{s=1}^{\beta}A_{s,s}$ be the sum of its diagonal blocks. Then,
\begin{equation*}
\lambda_{1+4 k}(H) \le \lambda_{1+k}( A+B)
\end{equation*}
 for all $k=0,\ldots,n-1$.
\end{cor}

\vskip 10pt
\begin{cor}\label{detail}
 Let $T\in\bM_{n}$ be Hermitian and let $\{S_i\}_{i=1}^\beta\in\bM_{n}$ be commuting Hermitian matrices with $\beta\in\{3,4\}$. Then,
 \begin{equation*}
\left\|\sum_{i=1}^\beta S_iT^2S_i\right\| \le \left\|\sum_{i=1}^\beta TS_i^2T\right\|
\end{equation*}
for all symmetric norms, and
\begin{equation*}
\lambda_{1+4 k}\left(\sum_{i=1}^{\beta} S_iT^2S_i\right) \le
 \lambda_{1+k}\left(\sum_{i=1}^{\beta} TS_i^2T\right)
\end{equation*} for all $k=0,\ldots,n-1$.
\end{cor} 

The proofs of these corollaries are quite similar of those of Section 2. We give details only for the norm inequality of Corollary \ref{detail}.

\vskip 10pt
\begin{proof} We may assume that $\beta=4$ by completing, if necessary with $S_4=0$. So, let $T\in\bM_n^+$ and let $\{S_i\}_{i=1}^4$ be    four commuting Hermitian matrices in $\bM_n$. Then
$$ H= XX^*=
\begin{bmatrix} TS_1\\ TS_2\\ TS_3\\ TS_4
\end{bmatrix}
\begin{bmatrix} S_1T & S_2T &S_3T  &S_4T
\end{bmatrix}
$$
is positive and partitioned into Hermitian blocks with diagonal blocks $TS_i^2T$, $1\le i\le 4$. Thus, from Theorem \ref{thm-quaternion}, for all symmetric norms,
$$
\| H\oplus H\| \le \left\| \left\{\sum_{i=1}^4 TS_i^2T\right\} \oplus \left\{\sum_{i=1}^4 TS_i^2T \right\}\right\|
$$
or equivalently
$$
\| H \| \le \left\| \sum_{i=1}^4 TS_i^2T\right\|
$$
Since $H=XX^*$ and $X^*X=\sum_{i=1}^4 S_iT^2S_i$, the norm inequality of Corollary \ref{detail} follows.
\end{proof}

\vskip 50pt

J.-C. Bourin,

Laboratoire de math\'ematiques,

Universit\'e de Franche-Comt\'e,

25 000 Besan\c{c}on, France.

jcbourin@univ-fcomte.fr

\vskip 10pt
Eun-Young Lee

 Department of mathematics,

Kyungpook National University,

 Daegu 702-701, Korea.

eylee89@ knu.ac.kr

\vskip 10pt
Minghua Lin

 Department of applied mathematics,

University of Waterloo,

 Waterloo, ON, N2L 3G1, Canada.

mlin87@ymail.com

\end{document}